\documentclass[11pt]{article}

\usepackage{amsthm,verbatim}
\usepackage{amsfonts}
\usepackage{bbm}
\usepackage[mathcal]{euscript}
\usepackage{amsmath}
\usepackage{mathtools}
\usepackage[all]{xy}

\newtheorem{theorem}{Theorem}[section]

\newtheorem{proposition}{Proposition}[section]
\newtheorem{lemma}{Lemma}[section]
\newtheorem{corollary}{Corollary}[section]
\newtheorem{definition}{Definition}[section]
\newtheorem{remark}{Remark}[section]

\newcommand{\A}{{\mathbf A}}
\newcommand{\CA}{{\rm QHO}}
\newcommand{\CCC}{{\mathbf a}\,} 
\newcommand{\BB}{{\mathbf a}^\dagger}
\newcommand{\QQ}{{\rm Q}}
\newcommand{\E}{{L}}
\newcommand{\Ll}{{\mathbb P}}
\newcommand{\NN}{{\mathbf N}}
\newcommand{\ctp}{ctp}
\newcommand{\tp}{tp}
\newcommand{\Stab}{Stab}
\newcommand{\kk}{\mathbbm k} 

\newtheorem{pkt}{}[subsection]  
\newcommand{\bpk}{\begin{pkt}\rm }  
\newcommand{\epk}{\end{pkt}}  
\newtheorem{Lemm}{Lemma}[subsection]  
   
\newtheorem{Prop}{Proposition} 
   
\newtheorem{Cor}{Corollary}

 \newcommand{\Hom}{{\rm Hom}}

\newcommand{\R}{{\mathbb R}}   
   
\newcommand{\Z}{{\mathbb Z}}   
\newcommand{\N}{{\mathbb N}}  
\newcommand{\C}{{\mathbb C}}  
  
\newcommand{\ee}{\end{equation}}  
\newcommand{\be}{\begin{equation}}

\newcommand{\bl}{\begin{Lemm}}   
   
\newcommand{\el}{\end{Lemm}}   
   
\newcommand{\bt}{\begin{Theo}}   
   
\newcommand{\et}{\end{Theo}}   
   
\newcommand{\bp}{\begin{Prop}}   
   
\newcommand{\ep}{\end{Prop}}   
   
\newcommand{\inv}{^{-1}}   
   
\newcommand{\bc}{\begin{Cor}}   
   
\newcommand{\ec}{\end{Cor}}   
  
\newcommand{\lb}{\label}  

\newcommand{\subs}{\subseteq}

\newcommand{\eps}{\epsilon}   
 
\newcommand{\XX}{{\rm H}}

\newcommand{\pp}{{\bf p}}

\newcommand{\PPP}{{\rm P}} 

\newcommand{\Aa}{\mathcal{A}}
\usepackage[mathcal]{euscript}

\DeclareMathOperator{\Gal}{Gal}
\DeclareMathOperator{\ACF}{ACF}

\DeclareMathOperator{\eq}{eq}
\DeclareMathOperator{\sep}{sep}
\DeclareMathOperator{\alg}{alg}

\title{Quantum Harmonic Oscillator as Zariski Geometry}
\author{Vinesh Solanki, Dmitry Sustretov and Boris Zilber}

\begin{document}

\maketitle


\section{Introduction and Background}
The present paper investigates an example of physical interest (the
quantum harmonic oscillator) using model-theoretic
methods. Specifically, we associate to this system a structure $\CA_N$
(dependent on the positive integer number $N$) on the universe $\E$
which is a finite cover of order $N,$ of the projective line
$\Ll=\Ll(\mathbb{F}),$ $\mathbb{F}$ an algebraically closed field of
characteristic $0.$ We prove that $\CA_N$ is a complete irreducible
Zariski geometry of dimension $1.$ We also prove that {\em $\CA_N$ is
  not classical} in the sense that the structure is not interpretable
in an algebraically closed field and, for the case $\mathbb{F}=\C,$ is
not a structure on a complex manifold.

There are several reasons that motivate our interest in this
particular example. Model-theoretically, very ample Zariski geometries
(as introduced in \cite{HZ}) give an abstract characterization of the
geometry of algebraic varieties whereas ample Zariski geometries can
only be shown to be finite covers of projective curves over
$\mathbb{F}$ for some algebraically closed field $\mathbb{F}$. In
\cite{HZ} a class of ample but not very ample Zariski geometries is
constructed; the construction involves certain non-abelian group
extensions and produces Zariski geometries which are acted upon by
these groups via Zariski automorphisms. The structure $\CA_N$ is an
example of an ample but not very ample Zariski geometry differing from
the examples in \cite{HZ}. In this regard, we hope that $\CA_N$ sheds
some new light on the ample/very ample distinction.

From a representation-theoretic perspective, the relevant algebra for
the analysis of the quantum harmonic oscillator is the Heisenberg
algebra; namely the algebra over $\mathbb{F}$ with generators $\PPP$
and $\QQ$ satisfying the relation \be \label{QP} [\QQ,\PPP] =
\QQ\PPP-\PPP\QQ=i,\ee where $i$ denotes a specific choice of
$\sqrt{-1}$ in $\mathbb{F}$. Note that this algebra is isomorphic to
the $\mathbb{F}$-algebra with generators $\NN, \BB, \CCC$ subject to
the relations
\[ [\NN,\BB] = \BB \qquad [\NN,\CCC] = -\CCC \qquad [\CCC,\BB] = 1 \] with isomorphism given by
\[ \NN \mapsto \frac{1}{2}(\PPP^2+\QQ^2) - \frac{1}{2} \qquad \CCC
\mapsto \frac{1}{\sqrt{2}}(\QQ + i\PPP) \qquad \BB \mapsto
\frac{1}{\sqrt{2}}(\QQ - i\PPP) \] The reader familiar with physics
will recognize $\BB$ and $\CCC$ as being creation and annihilation
operators respectively and $\NN$ as being the number operator of the
system (Hamiltonian minus $1/2$). The structure $\CA_N$ has a quotient
$Q_{N}$ which is a bundle of eigenspaces of $\NN$ and will be seen to
consist virtually of an uncountable collection of irreducible
representations of the Heisenberg algebra $A$ (each of countably
infinite dimension over $\mathbb{F}$) which are neither highest nor
lowest weight.

We proceed to give a summary of this paper. In section 2, the relevant
structures $\CA_N$ and $Q_{N}$ are defined. We demonstrate that for
even $N$, the theory of $\CA_N$ is categorical in uncountable
cardinals. Sections 4 and 5 are devoted to the aforementioned
non-interpretability proofs, after the introduction of appropriate
background material on Galois cohomology. In section 6, we carry out
an analysis of definable sets in $Q_{N}$ using the methods of
\cite{Zil2}. In particular, it is shown that the structure $Q_{N}$ is
a Zariski structure (according to the definition of \cite{Zil1}). It
should be noted that the algebra $A$ differs from the class of {\em
  quantum algebras at roots of unity} considered in \cite{Zil2}
insofar as irreducible modules for $A$ are necessarily
infinite-dimensional. Consequently, this paper constitutes an
extension of the construction and method of proof in \cite{Zil2} to a
wider class of noncommutative algebras, although it is not clear at
present precisely what this class of noncommutative algebras
is. Finally, it follows that $\CA_N$ is a Zariski geometry by virtue
of the two structures $\CA_N$ and $Q_{N}$ being bi-interpretable.

Some remarks about the methods of noncommutative geometry (algebraic
or \'a la Connes) and possible interactions with model theory are in
order. A feature of noncommutative geometry is that concrete
constructions of geometric counterparts to the algebras studied isn't
carried out, the algebraic (or $C^{*}$-algebraic) methods in
themselves sufficing for any ``geometric" arguments one may need to
produce. It is our belief, that though these methods are very powerful
in their own right, the absence of geometric counterparts
corresponding to (noncommutative) algebras results in a picture that
is incomplete. This paper, and the paper \cite{Zil2} represent steps
taken in the direction towards filling this gap. Furthermore, given
that the notion of a Zariski geometry provides an abstract
characterization of the geometry on an algebraic variety,
model-theoretically one has the means of proving that a non-classical
geometric structure associated to a specific noncommutative algebra is
suitably algebro-geometric.

It should be noted that the Heisenberg algebra is a $*$-algebra: it is
an algebra equipped with an additional operation $*$ which associates
to each element $X$ of the algebra an element $X^{*}$, seen (by
analogy with Hilbert space theory) as the {\em adjoint} to $X$. It is
not a $C^{*}$-algebra: any representation of the Heisenberg algebra as
an algebra of operators on a Hilbert space must, by the nature of the
defining relations, result in at least one of the operators being
unbounded. Of course, an important theorem in the representation
theory of $C^{*}$-algebras is that any $C^{*}$-algebra can be
represented as an algebra of {\em bounded} operators on some Hilbert
space. Consequently, one does not have many of the methods available
to non-commutative geometers to study this algebra directly. The {\em
  Weyl algebra} (the `exponential' of the Heisenberg algebra) is a
$C^{*}$-algebra and is consequently the favoured object of study. What
is interesting about the approach developed in this paper is that this
apparent issue with the Heisenberg algebra has not manifested itself
geometrically: the corresponding geometry is still rich.

By the postulates of quantum mechanics, $\PPP$ and $\QQ$ are
considered to be self-adjoint (self-adjoint operators have real
eigenvalues). Consequently, so is $\XX$. The Zariski structure
considered does not originally witness the $*$-structure on the
Heisenberg algebra, and so it produces, for $\mathbb{F}=\C,$
essentially a (non-classical) complex geometry. The assumption of
self-adjointness, in the canonical commutative context, leads to {\em
  cutting out the real part} of a complex variety.  The result of the
same operation with our structure $\CA$ is the discrete substructure,
(the finite cover of) the natural numbers $\mathbb{N}$. From a
physical standpoint, this leads to the quantization of the energy
levels of the Hamiltonian. Our Zariski geometry could therefore be
seen as the complexification of this discrete structure.

\section{The structures $\CA_N$ and $Q_{N}$} 

\begin{definition}
\label{defn:qho-struct}
 We consider the two-sorted theory $T_{N}$ with sorts $L$ and $\mathbb{F}$ in the language $\mathcal{L} = \mathcal{L}_{r} \cup \{\infty, \pi, \cdot, \A, \A^{\dagger}\}$ subject to the following axioms:

\begin{enumerate}
\item $\mathbb{F}$ is an algebraically closed field of characteristic $0$. 
\item $\mathbb{P}$ is the projective line over $\mathbb{F}$. 
\item $\pi: L \rightarrow \mathbb{P}$ is surjective. 
\item Let $\mu_{N}$ denote the group of $N$-th roots of unity in $\mathbb{F}$. We have a free and transitive group action $\cdot: \mu_N \times L \rightarrow L$ on each of the fibers $\pi^{-1}(x)$ for $x \in \mathbb{P}$. 
\item\label{item:phs-prod} The ternary relations $\A$, $\A^{\dagger}$ (on $L^{2} \times \mathbb{F}$) obey the following property:

\[ (\forall a \in \mathbb{F})(\forall e \in \pi^{-1}(a))(\exists b \in \mathbb{F})(\exists e^{'} \in \pi^{-1}(a +1))(b^{2} = a \wedge 
\A^{\dagger}(\gamma \cdot e, \gamma \cdot e^{'}, b) \wedge \A(\gamma \cdot e^{'}, \gamma \cdot e,b)) \] for every $\gamma \in \mu_N$.
\item For $N$ even, we postulate the following additional properties for 
$\A, \A^{\dagger}$: 

\[ \A^{\dagger}(e, e^{'}, b) \rightarrow \A(\gamma \cdot e, - \gamma \cdot e^{'}, -b)\]
\[ \A(e^{'}, e,b) \rightarrow \A^\dagger(\gamma \cdot e^{'}, - \gamma \cdot e, -b) \]
\end{enumerate}
\end{definition} The significance of the relations $\A,\A^{\dagger}$ will become apparent from Remark \ref{rem: well-def} below. Evidently the theory $T_{N}$ is first-order. We will denote models of $T_{N}$ by $\CA_N$. It is clear, by the freeness and transitivity of the action of $\mu_N$, that each fiber has size $N$. Thus any model $\CA_N$ is a finite cover of the projective line over some algebraically closed field $\mathbb{F}$. Fix a model $\CA_N$ and take the quotient of $L \times \mathbb{F}$ by the equivalence relation  
\[ (e,y) \sim (e',y') \Leftrightarrow (\exists \gamma \in \mu_N)(\gamma \cdot e = e^{'} 
\wedge \gamma^{-1}y = y^{'})\] We shall denote this bundle by $Q_{N}$. 

\begin{remark} Let $[(e,y)]$ denote the equivalence class of $(e,y) \in L \times \mathbb{F}$. The set $\mathcal{H}_{x} :=\{ [(e,y)]: \pi(e) = x\}$ is a definable $\mathbb{F}$-vector space of dimension $1$. Moreover, scalar multiplication is compatible with the action of $\mu_N$ on $L$. 
\end{remark}

\begin{proof} We abuse notation by denoting the equivalence class $[(e,y)]$ by $(e,y)$. Define
\[ \lambda(e,y):=(e,\lambda y) \qquad (e,y_{1}) + (e,y_{2}):=(e, y_{1} + y_{2}) \qquad y_{1}, y_{2} \in \mathbb{F}, x \in \mathbb{P}\] These are easily verified to be well-defined. Compatibility with the action of $\mu_N$ follows from the observation that $(\gamma \cdot e, y) = (e, \gamma y) = \gamma (e,y)$ for $\gamma \in \mu_N$.
\end{proof} Thus the universe of $Q_{N}$ consists of a trivial line bundle $\mathcal{H} = \bigcup_{x \in \mathbb{P}} \mathcal{H}_{x}$. We define linear maps $\CCC$, $\BB$ (induced from the relations $\A, \A^{\dagger}$) by
\[ \BB(e,1) := (e^{'},b) \qquad \CCC(e^{'},1) := (e,b) \] and extending linearly, where $\A^\dagger(e, e^{'},b)$ (and $\A(e',e,b)$) hold in the structure $\CA_N$. 

\begin{remark} \label{rem: well-def} The maps $\CCC, \BB$ are well-defined. 
\end{remark}

\begin{proof} Suppose we have that $\mathbf{a}^{\dagger}(e,1) = (e^{'}, b)$. 
Then for any $\gamma \in \mu_N$, it follows that 
\[ \mathbf{a}^{\dagger}(\gamma \cdot e, 1) = \gamma \mathbf{a}^{\dagger}(e,1) = \gamma(e^{'}, b) = 
(\gamma \cdot e^{'}, b) \] but we already have that $\A^{\dagger}(\gamma \cdot e, \gamma \cdot e, b)$ in the structure $\CA_N$. If $N$ is even then $-1 \in \mu_N$ and 
\[ \mathbf{a}^{\dagger}(-e, 1)= \mathbf{a}^{\dagger}(e,-1) = (e^{'}, -b) \] and $(-b)^{2} = a$. But we also have that $\A^{\dagger}(e,-e',-b)$ holds in $\CA_N$. Analogously for $\CCC$. 
\end{proof} 

\begin{lemma} \label{lemma:qho-Rdef} Let $\mathbb{F} = \mathbb{C}$. Then for each $N$, one can construct $\CA_N$ definable in $\mathbb{R}$. 
\end{lemma}

\begin{proof} Consider $\C$ as $\R+i\R,$ definable in $\R.$
  Choose an $\R$-definable complex
function $x \mapsto x^{\frac{1}{2}}$ satisfying
$(x^{\frac{1}{2}})^2=x$. Define $\E$ to be $\mathbb{P}^{1}(\C) \times \mu_N$ and $\pi$
to be the projection on the first factor. Then one can define the
relation $\A^{\dagger}$ to hold on points $((x,\alpha), (x+1,\beta), y)$
if 
$$
y = \alpha\beta^{-1}x^{\frac{1}{2}}
$$
\noindent The relation $\mathbf{A}$ is definable in terms of $\A^{\dagger}$. These relations are readily seen to be definable in $\R$.
Let us also note that the same definition works for any real closed
field.
\end{proof}

\begin{proposition} \lb{prop1} The theory $T_N$ is consistent and, for even $N,$ is categorical in uncountable
cardinals. Moreover,
if $\mathbb{F}$ and $\mathbb{F}'$ correspond to the field sort in two models
$\CA$ and $\CA'$ of theory $T_N$ and there exists $i:\mathbb{F} \to \mathbb{F}',$  
a ring isomorphism, then $i$ can be extended to an isomorphism
$\hat i: \CA\to \CA'.$ In particular the only relations on $\mathbb{F}$ induced
from $\CA$ are the initial relations corresponding to the field structure.
\end{proposition}
\begin{proof} Consistency is immediate from Lemma \ref{lemma:qho-Rdef}. To prove categoricity, we may assume that $\mathbb{F}=\mathbb{F}'$ and $i$ is the identity. Partition $\Ll$ into the orbits of the action of
the additive subgroup $\Z\subs \mathbb{F}$:
$$\Ll=\bigcup_{s\in S}s+\Z$$
where $S$ is some choice of representatives, one for each orbit ($\infty+m=\infty$ for each $m\in \Z,$ so we have a one $1$-element orbit).
For each $s\in S$ choose first $e_s\in \E(\CA)$  and $e'_s\in \E(\CA').$
Now for  
 each $n\in \Z$ choose arbitrarily $(s+n)^{\frac{1}{2}}.$ By the axioms
there is $e\in \E\cap \pi\inv(s+1)$ such that $\A^{\dagger}(e_{s},e,\eps (s+1)^{\frac{1}{2}})$ holds for some $\eps\in \{ 1,-1\}.$  
Define $e_{s+1}:=\eps e$ which is in $\E$ since $N$ is even and
$\eps\in \mu_N.$ Then $\A^{\dagger}(e_{s},e_{s+1},(s+1)^{\frac{1}{2}})$ also holds by the axioms. Analogously define $e'_{s+1}\in \E(\CA').$ By induction 
we can define $e_{s+n}$ and $e'_{s+n}$ for all $n\ge 0$ so that $\A^{\dagger}(e_{s+n},e_{s+n+1},(s+n)^{\frac{1}{2}})$ (respectively $\A^{\dagger}(e'_{s+n}, e'_{s+n+1},(s+n)^{\frac{1}{2}})$) holds in $\CA$ (respectively $\CA'$). 

By a similar inductive procedure for all $n> 0$ define $e_{s-n}$ so that $\A(e_{s-n+1},e_{s-n},(s-n)^{\frac{1}{2}})$ holds in $\CA$ (with analogous relations for the second model). We then extend $i$ by mapping $e_{s+n} \mapsto e'_{s+n}$ for all $n \in \mathbb{Z}$. Repeating for each representative in $S$ defines an isomorphism $\CA \rightarrow \CA_N$ extending $i$. 
 \end{proof} We now discuss $\CA_N(\mathbb{C})$ under the additional assumption that $\PPP$ and $\QQ$ are self-adjoint. By virtue of the relation $\NN = \BB\CCC = \frac{1}{2}(\PPP^{2} + \QQ^{2}) - \frac{1}{2}$, one sees that $\NN$ is also self-adjoint; hence must have real non-negative eigenvalues. The only points in
$\Ll$ that survive this extra condition are the non-negative integers $\N.$ The 
corresponding points in $\E$ form the $N$-cover of $\N,$ so we get the discrete
structure
$\CA_N(\N)$ as the real part of $\CA_N(\C).$ Conversely,
the latter is the complexification of the former. We close this section with the following easy observation.

\begin{remark} \label{rem: bi-interp} $\CA_N$ and $Q_N$ are
  bi-interpretable. Indeed, elements of the form $[(e,1)]$ in
  $\mathcal{H}$ comprise the sort $\E$, $\A^{\dagger}((e,1),(e',1),b)$ holds
  if and only if $\BB[(e,1)] = [(e',b)]$. Similarly for $\A$.
\end{remark}

\section{The Weil-Ch\^atelet group}

In order to show that the structure $\CA_N$ is not definable in an
algebraically closed field, we need to recall some facts about
principal homogeneous spaces under a group $A$, their description in
terms of Galois cohomology, and the group structure on their
isomorphism classes in the case $A$ is Abelian.

\begin{definition}[Principal homogeneous space]
  Let $A$ be a group. A \emph{principal homogeneous space over $A$ (or
    torsor)} is a set $X$ with a free transitive action of $A$ on $X$.
\end{definition}

This notion can be of course internalised to various categories. We
will be interested in the following case: if $A$ is an algebraic group
defined over a field $k$, the set $A(K)$ of $K$-points for any Galois
extension $K/k$ is acted upon by the Galois group $\Gal(K/k)$. If $X$
is a variety defined over $k$ which is a acted upon freely and
transitively by $A$ then the set $X(K)$ of $K$-points of $X$ is a
principal homogeneous space over $A(K)$, and moreover the action of
$A$ is compatible with the action of $\Gal(K/k)$.

\begin{definition}[Principal homogeneous space over a $G$-group]
  Let $G$ be a group. A \emph{$G$-set} is a set endowed with an action
  of $G$. A \emph{$G$-group} $A$ is a group that is endowed with an
  action of $G$ which is compatible with the group operation: $g (a
  \cdot b) = ga \cdot gb$. A principal homogeneous space over $A$ is a
  $G$-set with a right action of $A$ which is free and transitive and
  is compatible with the action of $G$: $g (x \cdot a) = x \cdot ga$.

  A morphism between principal homogeneous spaces under the same
  $G$-group $A$ is a map that preserves the action by $A$.
\end{definition}

The principal homogeneous spaces over a $G$-group can be classified up
to  isomorphism using group cohomology (we refer to \cite{serre-galois}
for detailed exposition). 

\begin{definition}[Zeroth and first group cohomology groups]
  Let $G$ be a group and let $A$ be a $G$-group. The group $H^0(G,A)$
  is defined as the subgroup of $G$-invariant elements of $A$: $\{ a
  \in A \mid \forall g \in G\ ga=a \}$. 

  The maps $h: G \to A, \sigma \mapsto h_\sigma$ that satisfy the
  condition $h_{\sigma\tau} = h_\sigma \sigma(h_\tau)$ are called
  \emph{cocycles}. Two cocycles $h,h'$ are called \emph{cohomologous}
  if there exists an element $a \in A$ such that
  \mbox{$h_\sigma = \sigma(a) h_\sigma a^{-1}$}.  For a general $A$
  one defines the first group cohomology set $H^1(G,A)$ to be the set
  of cocycles modulo the equivalence relation of being
  cohomologous. In case $A$ is Abelian, the elementwise product of
  cocycles is a cocycle and they form an Abelian group, moreover, the
  set of cocycles cohomologous to the zero cocycle forms a subgroup,
  so $H^1(G,A)$ has the natural group structure.
\end{definition}

Given a principal homogeneous space $P$ one constructs a cocycle as
follows: choose a point $p \in P$ and act on it by an element $\sigma$
of the group $G$. Since the action of $A$ on $P$ is transitive and
free there is a unique element of $A$ that maps $p$ to $\sigma \cdot
p$, so let $h_\sigma$ be this element. Then $h$ is a cocycle. If one
takes another point $q \in P$ then one obtains a cocycle $h'$ of the
form $h'_\sigma = (\sigma \cdot b) h_\sigma a^{-1}$ where $b$ is the
element of $A$ such that $b p = q$.

This defines a map from the set of isomorphism classes of principal
homogeneous spaces over a $G$-group $A$ to $H^1(G,A)$. An inverse map
is straightforwardly defined: for a cocycle $h$ make $A$ into a
principal homogeneous space by twisting action of $G$, $\sigma * a =
h_\sigma (\sigma \cdot a)$ (the cocycle condition ensures that this
indeed defines an action).

When $A$ is an algebraic group, one is interested in recovering the
principal homogeneous space structure as an algebraic, not just
abstract, action of $A$.

\begin{proposition}
  Let $K/k$ be a finite Galois extension. Then the set of
  $k$-isomorphism classes of principal homogeneous $A$-spaces defined
  over $K$ is in bijective correspondence with elements of \linebreak
  $H^1(\Gal(K/k), A(K))$.
\end{proposition}

\begin{proof}
  If $K/k$ is a finite Galois extension then it has a primitive
  element, so that $K = k(\alpha)$. Let $f(x) \in k[x]$ be the minimal
  polynomial of $\alpha$ and let $X$ be the affine 0-dimensional
  variety defined by it. The $K$-points (or, which is the same,
  $k^{\alg}$-points) of $X$ are acted upon freely and transitively by
  $\Gal(K/k)$.  In order to construct a principal homogeneous space
  for $A$ corresponding to the cocycle $h$, pick a point $x_0 \in X$,
  consider the space $A \times X$ and the action of $\Gal(K/k)$ given
  by
$$
\sigma \cdot (a, \eta \cdot x_0) = (\eta(h_\sigma) \cdot a, \sigma
\eta \cdot x_0)
$$
\noindent It is a well-known fact that the quotient of a
quasi-projective variety by the finite group action is a
quasi-projective variety, therefore there exists a quotient of $A
\times X$ be the action we have just described, call it $Y$. One
checks that the $K$-points of $Y$ are an abstract principal
homogeneous space corresponding to the cocycle $h$. As an alegbraic
variety, $Y$ is invariant under the action of $\Gal(K/k)$, so is
defined over $k$.
\end{proof}

In case $A$ is Abelian, there is a group structure on the
$k$-isomorphism classes of principal homogeneous $A$-spaces defined
over $K$. This group is called \emph{Weil-Ch\^atelet group}, denoted
$WC(K/k,A)$.

The group law was defined by Ch\^atelet for elliptic curves and for
arbitrary Abelian varieties by Weil in \cite{weil-phs}.  Weil's
construction essentially uses the theorem on birational group laws
(also known as ``Weil's group chunk theorem''). In case $A$ is
0-dimensional, this technology in unnecessary and the construction
proceeds as follows.

Let $P_g, P_h$ be principal homogeneous spaces over the same Abelian
0-dimensional group $A$ that correspond to cocycles $g,h \in
H^1(\Gal(K/k), A(K))$. Consider the following action of $\Gal(K/k)$ on
$P_g \times P_h$:
$$
\sigma \cdot (a, b) = (\sigma \cdot a, \sigma^{-1} \cdot b)
$$
\noindent and let $Z$ be the quotient under this action. Points of $Z$
correspond to orbits of the action. Define the action of $\Gal(K/k)$
on $Z$: let the orbit with a representative $(a,b)$ be mapped by
$\sigma$ to the orbit with the representative $(\sigma \cdot a,
b)$. Since $A$ is Abelian, this is well defined: $(\sigma \cdot a, b)$
is equivalent to $(\sigma \cdot (\eta \cdot a), \eta \cdot b)$ for any
$\eta \in \Gal(K/k)$.

\begin{proposition}
\label{prop:wc}
  The Weil-Ch\^atelet group $WC(K/k,A)$ is isomorphic to $H^1(\Gal(K/k),
  A(K))$.
\end{proposition}

\begin{proof}
  We need to ensure that the product of principal homogeneous spaces
  corresponds to the product of cocycles. 

  Let $P_g, P_h$ be principal homogeneous spaces corresponding to
  cocycles $g, h$ respectively and let $x \in P_g(K), y \in P_h(K)$ be
  such that for all $\sigma \in \Gal(K/k)$
$$
\begin{array}{l}
g_\sigma = a \textrm{ such that } a x = \sigma x\\
h_\sigma = a \textrm{ such that } a y = \sigma y
\end{array}
$$
Consider the element of $P_{g \cdot h}$ which is the image of
$(x,y)$. Then $(g \cdot h)_\sigma = g_\sigma \cdot h_\sigma$ since
$\sigma \cdot (x,y) = (g_\sigma x, h_\sigma y)$ and $(g_\sigma x,
h_\sigma y) \sim (g_\sigma h_\sigma x, y)$.
\end{proof}

Last remark we are going to make is that everything that has been said
above generalizes to non-finite Galois extensions and absolute Galois
groups, once group cohomology is replaced with profinite group
cohomology.

\begin{definition}[Group cohomology for profinite groups]
  A \emph{profinite group} is a topological group which is an inverse limit
  of finite discrete groups.

  Let $G$ be a profinite group and let $A$ be a topological $G$-group,
  i.e. a group endowed with a continuous action of $G$. The first
  cohmology set $H^1(G,A)$ is defined to be the set of all
  \emph{continuous} cocycles $h: G \to A, h \mapsto h_\sigma$ modulo
  the equivalence relation of being cohomological. In case $A$ is
  discrete this amounts to the requirement that the cocycle map
  factors through a finite quotient of $G$.

  The definition of the zeroth cohomology group is the same as in the
  case when $G$ is finite.
\end{definition}

\begin{proposition}[\cite{serre-galois}, Ch.I, \S 2, Proposition~8]
  Let $G$ be profinite and $A$ be a discrete $G$-group. Then
 $$
H^1(G,A) = \underleftarrow\lim H^1(G/U,A^U)
$$
\noindent where $U$ runs through all open normal subgroups and $A^U$
is the subgroup of $A$ fixed by the action of $U$.
\end{proposition}

One sees easily that the group $H^1(\Gal(k^{\sep}/k),
A(k^{\mathrm{sep}}))$ classifies principal homogeneous subspaces over
$A$ defined over any Galois extension of $k$.

\section{Non-definability of $\CA_N$ in the theory $\ACF_0$}

Let $M$ be a structure in a language $\mathcal{L}$ and $N$
be a structure in a language $\mathcal{L}'$. An \emph{interpretation}
of $M$ in $N$ is a structure with the universe being a definable set
(further denoted $M(N)$) in $N^{\eq}$ and such that the predicates of
$\mathcal{L}$ are definable relations in $N^{\eq}$ (further denoted
$P(N)$ where $P$ is a predicate of $\mathcal{L}$) such that $M(N)$ is
isomorphic to $M$ as an $\mathcal{L}$-structure. The notation $M(N)$
will be used to denote both the definable set and the structure with
such universe. If $X$ is definable set in $M^k$ then the image of it
under the isomorphism is denoted $X(N)$.

 If moreover $M(N)$ is a definable subset of $N^k$ for some $k$, one
 says that $M$ is \emph{defined} in $N$.

 \begin{remark}
   We adopt a similar notation for realisations of definable sets. If
   $X$ is a definable set and $M$ is a model then $X(M)$ will mean the
   set of $M$-tuples that belong to the realisation of $X$ in $M$.
 \end{remark}

 Let $(F,+,\cdot,0,1)$ be an infinite field and let $K$ be an
 algebraically closed field. Let $F$ be interpreted (or rather
 defined, since $ACF_p$ has elimination of imaginaries) in $K$. Then
 $F(K)$ in a definable set in the field $K$ with definable field
 operations.

\begin{theorem}[\cite{pillay-acf}, Theorem
  4.13; \cite{poizat-groups},Theorem 4.15 ]
\label{thm:field} 
There is a bijection between $F(K)$ and $K$, definable in $K$, which
is a field isomorphism.
\end{theorem}

\begin{theorem}[Kummer theory, \cite{birch-cycl}]
\label{thm:kummer}
Let $k$ be a perfect field and let $n$ be an integer that does not
divide the characteristic of $k$. Suppose $k$ contains $n$-th roots of
unity. Then
$$
H^1(\Gal(k^{\sep}/k), \mu_n) \cong \Hom(\Gal(k^{\sep}/k), \mu_n) \cong
k^\times/(k^\times)^n
$$
\end{theorem}

\begin{remark}
\label{rem:kummer-funct}
The correspondence $G \mapsto H^1(G,A)$ where $A$ is a trivial module
is actually a functor, i.e. to any homomorphism \mbox{$f: G_1 \to
  G_2$} corresponds a homomorphism \linebreak \mbox{$H^1(f):
  H^1(G_2,A) \to H^1(G_1, A)$}. In particular, given an automorphism
$\sigma: k \to k$, one can lift it to an automorphism of the separable
closure. The conjugation by $\sigma$ gives an automorphism
$\bar\sigma: \Gal(k^{\sep}/k) \to \Gal(k^{\sep}/k)$. One can check
that in the Kummer theory setting the induced map on cohomology is
$$
H^1(\bar\sigma): k^\times/(k^\times)^n \to k^\times/(k^\times)^n \ \
\ \ x \mapsto \sigma(x)
$$
\end{remark}

\begin{theorem}
  \label{thm:not-defn-acf}
  $\CA_N$ is not definable in an algebraically closed field.
\end{theorem}

\begin{proof} 
  Suppose that $\CA_N$ is definable in a field $K$ over a set of
  parameters $\kk$ which we can assume without lack of generality to
  be a subfield.  By Theorem~\ref{thm:field}, $\mathbb{F}(K)$ is
  definably isomorphic to $K$ as a field, so may just as well suppose
  that $\mathbb{F}(K)$ is interpreted as $\mathbb{A}^1_K$ with field
  operations given by that of $K$.

  Let $x$ be a generic point of the affine line $\mathbb{A}^1_K$ over
  $\kk$. Let $K'$ be an algebraically closed field that contains $x$ and
  $\kk$ (for example, one can take the algebraic closure of $\{x\} \cup
  \kk$ in the monster model).

  By definition of the structure $\CA_N$, for any $a \in K'$, the set
  $L_a=p^{-1}(a)$ is the principal homogeneous space over the group of
  $N$-th roots of unity, $\mu_N$.  The definable set $L_x$ is a
  principal homogeneous space over $\mu_N$ defined over $k=\kk(x)$
  (since $x$ is generic and hence transcendental over $k$). By
  Theorem~\ref{thm:kummer} a cocycle class in $H^1(\Gal(k^{\sep}/k),
  \mu_N) \cong k^\times/(k^\times)^N$ that describes $L_a$ is
  represented by some rational function $f(x) \in \kk(x)$ which
  without loss of generality can be assumed to be a polynomial since
  every coset of $(\kk(x)^\times)^N$ in $\kk(x)^\times$ contains one.

  By Remark~\ref{rem:kummer-funct}, a cocycle class corresponding to
  $L_{x+1}$ is $\sigma(f)$ where $\sigma: k \to k$ is the
  automorphism that sends $x$ to $x+1$, i.e. the corresponding
  polynomial is $f(x+1)$.

  Now notice that by definition of $\mathbf{A}^\dagger$
  (clause~\ref{item:phs-prod} in Definition~\ref{defn:qho-struct}),
  the principal homogeneous space $\sqrt{x} \mu_N(K')$ is the product
  of the principal homogeneous space $L_x(K')$ and the principal
  homogeneous space opposite to $L_{x+1}$. Then by
  Proposition~\ref{prop:wc} we have
$$
\frac{f(x)}{f(x+1)} = x^{N/2} \textrm{ mod } (k^\times)^N
$$
  \noindent Define the homomorphism $\mathrm{deg_N}:
  k^\times/(k^\times)^N \to \mathbb{Z}/N\mathbb{Z}$ to be the degree
  of the nominator minus the degree of denominator modulo
  $N$. Clearly, $\mathrm{deg}_N(\frac{f(x)}{f(x+1)}) = 0$ for any
  $f(x)$ and $\mathrm{deg}_N(x^{N/2})=N/2$, therefore the equality
  cannot hold  and we have come to a contradiction.   
\end{proof}
        
It is natural to ask  whether $\CA_N$ is also definable in the
structure of compact complex spaces (\cite{z-ccm}, \cite{pillay-ccm}).
We denote by $\mathcal A$ the multi-sorted structure where each sort
is the set of points of a compact complex space, and the language
consists of predicates corresponding to all complex analytic
subvarieties of all possible finite products of sorts. Let us pick
some very saturated elementary extension of $\mathcal A$ and denote it
$\mathcal A'$. We refer to \cite{moosa-nonstandard} for some basic
facts about this structure and a dictionary between model-theoretic
and complex geometric notions.

\begin{theorem}[\cite{moosa-nonstandard}]
\label{thm:ccm-field}   
  Let $F$ be an algebraically closed field of characteristic 0. If $F$
  is definable in $\mathcal A'$ then $F(\mathcal A')$ is definably
  isomorphic to $\mathbb{C}(\mathcal A')$.
\end{theorem}

\begin{theorem}
  \label{thm:not-defn-ccm}
$\CA_N$ is not definable in the structure $\mathcal A'$.
\end{theorem}

\begin{proof}
  Suppose the contrary, so $\CA_N$ is defined in $\mathcal A'$ over
  some set of parameters $B$. Then by Theorem~\ref{thm:ccm-field}
  $\mathbb{F}(\mathcal A')$ is definably isomorphic to
  $\mathbb{C}(\mathcal A')$, and we can without loss of generality
  suppose that $\mathbb{F}(\mathcal A')$ is interpreted by the field
  $\mathbb{C}(\mathcal A')$.

  Pick a point $x$ in $\mathbb{C}(\mathcal A')$ generic over $B$. As
  in the proof of Theorem~\ref{thm:not-defn-acf} denote $L_x$ and
  $L_{x+1}$ the definable sets $p^{-1}(x)$ and $p^{-1}(x+1)$
  respectively. Notice that $L_x(\mathcal A')$ and $L_{x+1}(\mathcal
  A')$ are internal to $\mathbb{C}(\mathcal A')$, i.e. there exist
  bijections between them and some principal homogeneous sets
  definable in the field $\mathbb{C}(\mathcal A')$ definable after
  adding some parameters (one only has to name a single point of $L_x$
  or $L_{x+1}$). It follows from the fact that every closed set of any
  Cartesian power of $\mathbb{P}^1(\Aa)$ is algebraic (by Chow's
  theorem; or one can derived this from the purity theorem for Zariski
  geometries), and from quantifier elimination that the field
  $\C(\Aa')$ is pure, i.e. has no additional definable sets but those
  that are definable in it as a field. Therefore, the proof of
  Theorem~\ref{thm:not-defn-acf} applies without modifications.
\end{proof}

\section{Definable sets in $Q_N$}

We can view $Q_{N}$ as a two-sorted structure $(\mathcal{H}, \mathbb{F})$, 
where $\mathcal{H}$ is defined as before. Introduce the projection map 
$\mathbf{p}: \mathcal{H} \rightarrow \mathbb{P}$ where 
$\mathbf{p}: (e,x) \mapsto \pi(e)$. Note that $\mathbf{p}$ is definable. We wish to pick `canonical basis' elements in each fiber $\mathcal{H}_{x}$ which we regard as having modulus one. In our terminology, these canonical basis elements are exactly the elements $(e,1)$ in each fiber. We introduce a (definable) predicate $\mathbf{E}(e,y)$ which says '$e$ is a canonical basis element of the fiber $\mathbf{p}^{-1}(y)$. 
\\
\\
We provide some motivation for the definable sets we wish to consider. Suppose that $v = (v_{1}, \dots, v_{s})$ is a tuple from the sort $\mathcal{H}$. We can re-index the $v_{i}$ according to the fibers of $\mathbf{p}$ in which they appear. Namely, we fix an enumeration $\{v_{ij}: 1 \leq i \leq t; 1 \leq j \leq s_{i}, \sum_{i} s_{i} = s \}$ so that given $v_{ij}, v_{kl}$, we have $i = k$ if and only if $\mathbf{p}(v_{ij}) = \mathbf{p}(v_{kl})$. By the axioms, we can find $a_{i} \in \mathbb{P}$ such that $\mathbf{p}(v_{ij}) = a_{i}$ for every $1 \leq i \leq t$. Moreover, because each $\mathbf{p}^{-1}(a_{i})$ is one-dimensional and we have basis elements $e_{i} \in \mathbf{E}(\mathcal{H}, a_{i})$, we can find scalars $\lambda_{ij} \in \mathbb{F}$ such that 
\[ \bigwedge_{i=1}^{t} \bigwedge_{j=1}^{s_{i}} \lambda_{ij} e_{i} = v_{ij} \] holds. Thus one expects all the sentences satisfied by $v$ to be determined by all the inter-relationships between the $e_{i}$. But the relationships between the $e_{i}$ depend on the orbits of the action of the additive subgroup $\mathbb{Z} \subseteq \mathbb{F}$ on $\mathbb{P}$. We set up some notation to describe these relationships. Suppose that $e_{i}$ and $e_{j}$ lie in the same orbit. Then $\pp(e_{j}) - \pp(e_{i}) = n$ for some $n \in \mathbb{Z}$; without loss $n$ is non-negative. Then there is a `path' in the structure connecting the fiber containing $e_{i}$ to the fiber containing $e_{j}$ via the operator $\BB$. We wish to construct an existential sentence $\theta_{ij}$ that codes this path. Writing $e_{i}^{0}$ for $e_{i}$, our candidate for $\theta_{ij}$ is the following:
\[\exists \gamma_{ij} \exists_{k=1}^{p} b_{ijk} \exists_{k=1}^{p} e_{i}^{k} \left( \begin{array}{l}\bigwedge_{k = 1}^{p} \mathbf{E}(e_{i}^{k}, \mathbf{p}(e^{k-1}_{i}) + 1) \wedge \bigwedge_{k=1}^{p} \BB e_{i}^{k-1} = b_{ijk} e_{i}^{k} \\  \wedge b^{2}_{ijk} = \mathbf{p}(e^{k}_{i}) \wedge e_{j} = \gamma_{ij} e^{p}_{i} \end{array} \right) \] This sentence is, of course, satisfied in $Q_N$. We now have enough information to construct a class of formulas with which to prove quantifier elimination.

\begin{definition} 
\label{definition - core formulas}
Let $\{v_{ij}: 1 \leq i \leq t; 1 \leq j \leq s_{i}, \sum_{i} s_{i} = s \}$ and $x = (x_{1}, \dots x_{r})$ be tuples of variables from the sorts $\mathcal{H}$ and $\mathbb{F}$ respectively. A \textbf{core formula} with variables $(v,x)$ is defined to be a formula of the following shape:

\[ \exists \lambda \exists_{i=1}^{t} e_{i} \exists_{i=1}^{t}y_{i} \exists \gamma \exists b\left( \begin{array}{ll} \bigwedge_{i=1}^{t} \bigwedge_{j =1}^{s_{i}} \pp(v_{ij}) = y_{i} \wedge \lambda_{ij} e_{i} = v_{ij} \wedge \mathbf{E}(e_{i}, y_{i}) \wedge \bigwedge_{(i,j) \in \Xi} \phi_{ij}(e_{i}, e_{j}, b_{ij}, \gamma_{ij}) \\ \wedge S(\lambda, y, \gamma, b, x) \end{array} \right) \] where

\begin{enumerate}
\item $\Xi$ is a subset of $\{(i,j): 1 \leq i,j \leq t\}$.
\item $y_{i} = y_{k}$ if and only if $i = k$. 
\item $S$ defines a Zariski constructible subset of $\mathbb{F}^{r_{1}} \times \mathbb{P}^{t} \times \mu_{N}^{r_{2}}$ where
\begin{enumerate}
\item $r_{1} = l(x) + l(b) + s$ (l denotes length)
\item $r_{2} = l(\gamma)$
\end{enumerate}
\item $\phi_{ij}$ is $\theta_{ij}$ with the existential quantification over $\gamma_{ij}, b_{ik}$ removed. 
\end{enumerate} A \textbf{core type} is defined to be a consistent collection of core formulas. If $(v,a)$ is a tuple of elements from $\mathcal{H}^{s} \times \mathbb{F}^{r}$ and $C \subseteq \mathbb{F}$ is a set of parameters, the \textbf{core type of $(v,a)$ over $C$}, denoted $\ctp (v,a/C)$, is defined to be the set of all core formulas with parameters from $C$ satisfied by $(v,a)$. 
\end{definition} 

\begin{proposition} 
\label{prop: core qe}
Let $Q_N$ be $\aleph_0$-saturated. Suppose that $(v,c), (w,d)$ are both tuples from $\mathcal{H}^{s} \times \mathbb{F}^{r}$ with the property that $\ctp(v,c/C) = \ctp(w,d/C)$. Then $\tp(v,c/C) = \tp(w,d/C)$. 
\end{proposition}

\begin{proof} We construct an automorphism $\tilde{\sigma}$ such that $\tilde{\sigma}: (v,c) \mapsto (w,d)$. The tuple $v$ is indexed as $\{v_{ij}: 1 \leq i \leq t; 1 \leq j \leq s_{i}, \sum_{i} s_{i} = s \}$ so that given $v_{ij}, v_{kl}$, we have $i = k$ if and only if $\pp(v_{ij}) = \pp(v_{kl})$. By what has already been discussed, the axioms provide us with:

\begin{enumerate}
\item Tuples $a_{i}^{1}$ such that $\pp(v_{ij}) = a_{i}^{1}$ for every $1 \leq i \leq t$. 
\item Basis elements $e_{i}^{1} \in \mathbf{E}(\mathcal{H}, a_{i})$ and scalars $\lambda_{ij}^{1}$ such that 
\[ \bigwedge_{i = 1}^{t} \bigwedge_{j=1}^{s_{i}} \lambda_{ij}^{1} e_{i}^{1} = v_{ij} \]
\end{enumerate} Now we construct the set $\Xi$ so that $(i,j) \in \Xi$ if and only if there is a path from $\pp^{-1}(a_{i}^{1})$ to $\pp^{-1}(a_{j}^{1})$ (coded by $\theta_{ij}$). Then the following conjunct holds:
\[ \bigwedge_{i=1}^{t} \bigwedge_{j=1}^{s_{i}} \pp(v_{ij}) = a_{i}^{1} \wedge \lambda_{ij}^{1}e_{i}^{1} = v_{ij} \wedge \mathbf{E}(e_{i}^{1}, a_{i}^{1}) \wedge \bigwedge_{(i,j) \in \Xi} \phi_{ij}(e^{1}_{i}, e^{1}_{j}, b^{1}_{ij}, \gamma^{1}_{i}) \] Denote the above formula by $\phi(v, e^{1}, a^{1}, \lambda^{1}, \gamma^{1}, b^{1})$. Consider the following set of formulas:
\[ \Sigma = \begin{array}{ll} \{\phi(w,e',x', \lambda', \gamma',b') \wedge S(x',\lambda',\gamma',b',d) : \\ \qquad Q_N \models \phi(v,e^{1},a^{1},\lambda^{1}, \gamma^{1},b^{1}) \wedge S(a^{1}, \lambda^{1},\gamma^{1},b^{1},c) \} \end{array} \] Here the variables have been primed to distinguish them from actual parameters. The $S$ range over all constructible subsets of an appropriate cartesian power of $\mathbb{F}$ with parameters from $C$. 
\\
\\
\textbf{Claim}: $\Sigma$ is consistent. 

\begin{proof} We show that $\Sigma$ is finitely consistent. By definition $\Sigma$ is closed under finite conjunctions, so let $\phi \wedge S \in \Sigma$. Then
\[ Q_N \models \phi(v,e^{1},a^{1},\lambda^{1},\gamma^{1},b^{1}) \wedge S(a^{1}, \lambda^{1},\gamma^{1},b^{1},c) \] Existentially quantifying out $e^{1},a^{1}, \lambda^{1}, \gamma^{1}$ and $b^{1}$, we obtain a core formula satisfied by $(v,c)$ over $C$. But $\ctp(v,c/C) = \ctp(w,d/C)$, so there are $e^{2}, a^{2}, \lambda^{2}, \gamma^{2}, b^{2}$ such that 
\[ Q_N \models \phi(w,e^{1},a^{1},\lambda^{1},\gamma^{1},b^{1}) \wedge S(a^{1}, \lambda^{1},\gamma^{1},b^{1},d) \] as required.
\end{proof} By saturation, the type $\Sigma$ is satisfied by a tuple $(e^{2}, a^{2}, \lambda^{2}, \gamma^{2}, b^{2})$ say. In particular, we have that
\[ \tp^{\mathbb{F}}(a^{1}, \lambda^{1}, \gamma^{1}, b^{1},c) = \tp^{\mathbb{F}}(a^{2}, \lambda^{2}, \gamma^{2}, b^{2}, d) \] and by saturation of $\mathbb{F}$ we therefore obtain an isomorphism $\sigma$ of $\mathbb{F}$ such that
\[ \sigma: (a^{1}, \lambda^{1}, \gamma^{1}, b^{1},c) \mapsto (a^{2}, \lambda^{2}, \gamma^{2}, b^{2}, d) \] It remains to extend $\sigma$ to the whole of $Q_N$; the proof proceeds as for categoricity. Partition $\mathbb{P}$ into orbits of $\mathbb{Z}$,
\[ \mathbb{P} = \bigcup_{x \in \Lambda} x + \mathbb{Z} \] for some set of representatives $\Lambda$. Given $x \in \Lambda$ suppose that only $a^{1}_{i_{1}}, \dots, a^{1}_{i_{q}}$ occur in the orbit $x + \mathbb{Z}$. Without loss, $x = a^{1}_{i_{1}}$ and the $a^{1}_{i_{j}}$ are listed in order of increasing distance from $x$. Now inductively extend as in the proof of Proposition \ref{prop1} noting that the construction automatically maps $e_{i_{j}}^{1} \mapsto e_{i_{j}}^{2}$ for every $1 \leq j \leq q$. 
\end{proof} It follows by compactness that every formula with parameters from $\mathbb{F}$ is equivalent to a boolean combination of core formulas. Some further analysis reveals the structure of subsets defined using parameters from both $\mathcal{H}$ and $\mathbb{F}$. These are determined by a class of formulas similar to core formulas, the only difference being that these formulas can also express information about how bases from the fibers containing these parameters from $\mathcal{H}$ are connected to other fibers via paths.

\begin{definition} 
\label{defn: gen core}
Let $e'$ be a tuple of elements from $\mathcal{H}$ with length $p$ such that all $e'_{i}$ are basis elements. Let $v = (v_{1}, \dots, v_{m})$, $w = (w_{1}, \dots, w_{n})$ be tuples of variables from $\mathcal{H}$. A \textbf{general core formula} with variables $(v,w,x)$ over $e'$ is defined to be a formula of the following shape:
\[ \exists \lambda \exists \mu \exists_{i=1}^{s} e_{i} \exists_{i=1}^{s} y_{i} \exists \gamma \exists b \left( \begin{array}{l} \bigwedge_{i=1}^{s} \bigwedge_{j=1}^{s_{i}} \pp(v_{ij}) = y_{i} \wedge \lambda_{ij}e_{i} = v_{ij} \wedge \mathbf{E}(e_{i}, y_{i}) \wedge \phi \wedge \bigwedge_{(i,j) \in \Xi} \phi_{ij} \\ \wedge S(\lambda, \mu, y, \gamma, b, x) \end{array} \right) \] where 
\begin{enumerate}
\item $\{v_{ij}: 1 \leq i \leq s, 1 \leq j \leq s_{i}\}$ is an appropriate enumeration of $v$
\item $\Xi \subseteq \{(i,j): 1 \leq i,j \leq s\}$
\item $S$ is a constructible subset of $\mathbb{F}^{r_{1}} \times \mathbb{P}^{s} \times \mu_{N}^{r_{2}}$ where
\begin{enumerate}
\item $r_{1} = l(x) + l(b) + m + n$
\item $r_{2} = l(\gamma)$
\end{enumerate}
\item $\phi$ is defined to be
\[ \bigwedge_{i=1}^{p} \bigwedge_{j=1}^{p_{i}} \mu_{ij} e'_{i} = w_{ij} \wedge \bigwedge_{(i,j) \in \Xi_{1}} \phi_{ij}(e'_{i}, e_{j}, b_{i}, \gamma_{ij}) \wedge \bigwedge_{(i,j) \in \Xi_{2}} \phi_{ij}(e_{i}, e'_{j}, b_{i}, \gamma_{ij}) \] where 
\[ \Xi_{1} \subseteq \{(i,j): 1 \leq i \leq p, 1 \leq j \leq s\} \qquad \Xi_{2} \subseteq \{(i,j): 1 \leq i \leq s, 1 \leq j \leq p \} \] and $\{w_{ij}: 1 \leq i \leq p, 1 \leq j \leq p_{i}\}$ is an appropriate enumeration of $w$. 
\end{enumerate}
We denote such a formula by $\exists e S$ and call $S$ the \textbf{Zariski constructible component} of $\exists e S$. 
\end{definition}

\begin{proposition} 
\label{prop: gen core qe} 
If $\phi$ is a formula with parameters from $\mathcal{H},\mathbb{F}$ then it is equivalent to a Boolean combination of general core formulas. 
\end{proposition}

\begin{proof} Suppose that $\phi(v,x)$ is a formula with free variables $(v,x)$ over a finite set of parameters $w = (w_{1}, \dots, w_{p})$ of $\mathcal{H}$ and some unspecified parameters from $\mathbb{F}$. Then $\phi(v,x)$ is equivalent to some $\phi_{1}(v,w,x)$ where $\phi_{1}(v,w',x)$ is a formula with free variables $(v,w',x)$ merely over some set of parameters from $\mathbb{F}$. Hence $\phi_{1}$ is equivalent to a boolean combination of core formulas over $\mathbb{F}$ by Proposition \ref{prop: core qe}. We show that every core formula is equivalent to a finite disjunction of general core formulas after substitution. 
\\
\\
So let $\varphi(v,w',x)$ be a core formula. We can fix an enumeration $\{v_{ij}: 1 \leq i \leq s, 1 \leq j \leq s_{i}, \sum_{i} s_{i} = n\}$ of $(v,w')$ such that 
\begin{enumerate}
\item $n$ is the length of $(v,w')$.
\item $\pp(v_{ij}) = \pp(v_{kl})$ if and only if $i = k$. 
\item Those $v_{ij}$ for which $v_{ij}$ is not in $w'$ for any $j$ are listed first, i.e. there is a maximum $m \leq s$ such that $v_{ij} \not \in w'$ for all $i \leq m$. 
\item For $i > m$, the $w'$ variables are listed last, i.e. there is a minimum $t_{i} \leq s_{i}$ such that $v_{ij} \in w'$ for all $j > t_{i}$. 
\end{enumerate} Now $\varphi(v,w',x)$ looks like
\[ \exists \lambda \exists_{i=1}^{t} e_{i} \exists_{i=1}^{t}y_{i} \exists \gamma \exists b\left( \begin{array}{ll} \bigwedge_{i=1}^{t} \bigwedge_{j =1}^{s_{i}} \pp(v_{ij}) = y_{i} \wedge \lambda_{ij} e_{i} = v_{ij} \wedge \mathbf{E}(e_{i}, y_{i}) \wedge \bigwedge_{(i,j) \in \Xi} \phi_{ij}(e_{i}, e_{j}, b_{ij}, \gamma_{ij}) \\ \wedge S(\lambda, y, \gamma, b, x) \end{array} \right) \] for some $\Xi \subseteq \{(i,j): 1 \leq i,j \leq s\}$ and $S$ over $k$. Substitute $w$ for $w'$. The resulting formula can be simplified by noting that some of the information it expresses is already contained in the theory. If $k > m$, then $y_{k} = \pp(w_{kl})$ is determined, thus $\exists y_{k}$ and such conjuncts can be dropped for $k > m$. Moreover, $\exists e_{k}$ can be dropped by replacing the formula with a finite disjunction, where each disjunct contains $e'_{k}$ for $e_{k}$ and $e'_{k}$ ranges over the finitely many canonical basis elements of $\pp^{-1}(y_{k})$. This allows us to make further deletions from each disjunct, namely $\mathbf{E}(e'_{k}, y_{k})$ (which trivially holds) and $\lambda_{kl}e'_{k} = w_{kl}$ for $l > t_{k}$ (because $\lambda_{kl}$ is determined), and we can therefore drop $\exists \lambda_{kl}$. This leaves us with the formula

\[  \bigvee_{\begin{array}{c} e' = (e_{t+1}, \dots, e_{s}) \\ e'_{k} \in \pp^{-1}(\pp(w_{kl})) \wedge \mathbf{E}(e'_{k}, \pp(w_{kl})) \end{array}} \exists \lambda \exists_{i=1}^{m} e_{i} \exists_{i=1}^{m}y_{i} \exists \gamma \exists b \left(\begin{array}{l} \bigwedge_{i=1}^{m} \bigwedge_{j=1}^{s_{i}} \pp(v_{ij}) = y_{i} \wedge \lambda_{ij}e_{i} = v_{ij} \\ \wedge \mathbf{E}(e_{i}, y_{i}) \wedge \phi'  \end{array} \right) \] for appropriate $\phi'$ which we now determine. Clearly in
\[ \bigwedge_{(i,j) \in \Xi} \phi_{ij}(e_{i}, e_{j}, b_{ij}, \gamma_{ij}) \] if we substitute the parameters $e'_{k}$ for $k > m$ then some conjuncts are eliminable; namely those $\phi_{kl}(e'_{k}, e'_{l}, b_{kl}, \gamma_{kl})$ for $k,l > m$ (the theory itself tells us about paths that connect the fibers containing these $e'_{k}, e'_{l}$). Hence the quantifiers $\exists b_{kl}$ and $\exists \gamma_{kl}$ can also be eliminated from each disjunct. Define the sets
 \[ \Xi_{1} = \{(i,j) \in \Xi: 1 \leq i \leq m, m < j \leq s \} \qquad \Xi_{2} = \{(i,j) \in \Xi: m < i \leq s, 1 \leq j \leq m \}  \] \[ \Phi = \{(i,j) \in \Xi: 1 \leq i,j \leq m \} \] Then we have $\phi'$ as the formula
 \[ \bigwedge_{i = m+1}^{s} \bigwedge_{j=1}^{t_{i}} \lambda_{ij}e'_{i} = v_{ij} \wedge \bigwedge_{(i,j) \in \Phi} \phi_{ij} \wedge \bigwedge_{i=1}^{2} \bigwedge_{(i,j) \in \Xi_{i}} \phi_{ij}  \wedge S'(\lambda, y, \gamma, b, x) \] where $S'$ is $S$ with the determined parameters $\lambda_{kl}, y_{k}, b_{kl}$ and $\gamma_{kl}$ substituted for the appropriate variables. Now re-label, putting $\mu_{ij} = \lambda_{i+m, j}$. We see that each disjunct is a general core formula as required. 
\end{proof}

\subsection{Constructibility} 
Proposition \ref{prop: gen core qe} suggests taking sets of the form $\exists e C$ (where $C$ defines a closed subset of a cartesian power of $\mathbb{F}$) as giving us the closed subsets of a topology on the sorts of $Q_N$ and their cartesian powers. Additional technicalities are required (suitably adapted from \cite{Zil2}) to deal with finite intersections of sets of the form $\exists e C$ (there is no a priori guarantee that they will still be of this form). Specifically, we require an analysis of how $C$ transforms under applications of $\mu_{N}$ to basis elements in the fibers.

\begin{definition}
\label{defn: mu inv}
Let $\exists e C$ be a general core formula with $C$ giving a closed subset of $\mathbb{F}^{r_{1}} \times \mathbb{P}^{s} \times \mu_{N}^{r_{2}}$. We define the \textbf{action} of $\delta \in \mu_{N}^{r_{2}}$ on $C$ to be 
\[ C^{\delta} = \{(\lambda_{ij}, \mu, y, \gamma, b, a): (\delta_{i}^{-1}\lambda_{ij}, \mu, y, \delta \cdot \gamma, b, a) \in C\} \] where 
\[ \delta \cdot \gamma = \left\{ \begin{array}{ll} \delta_{i}^{-1} \gamma_{ij} \delta_{j} & (i,j) \in \Xi \\ \gamma_{ij} \delta_{j} & (i,j) \in \Xi_{1} \\ \delta_{i}^{-1}\gamma_{ij} & (i,j) \in \Xi_{2} \end{array} \right. \] $C$ is defined to be \textbf{$\mu_{N}$-invariant} if $C^{\delta} = C$ for every $\delta \in \mu_{N}^{r_{2}}$. 
\end{definition}

\begin{lemma}
\label{lem: construct}
If $\phi$ is a formula with parameters from $\mathcal{H}, \mathbb{F}$ then it is equivalent to a Boolean combination of general core formulas with Zariski constructible component of the type $\exists e (C_{1} \wedge \neg C_{2})$ where $C_{1}, C_{2}$ are Zariski closed and $C_{2}$ is $\mu_{N}$-invariant. 
\end{lemma}

\begin{proof} The same as \cite{Zil2} which we restate in our case. Fix a tuple $\alpha \in \mathbb{F}$ and recall that $tp^{\mathbb{F}}(\alpha)$ denotes the type of $a$ in the language of fields. Put
\[ \Sigma(\alpha) = \{C_{1} \wedge \neg C_{2}: Q_N \models (C_{1} \wedge \neg C_{2})(\alpha) \mbox{ and $C_{1}$, $C_{2}$ are $\mu_{N}$-invariant} \} \] It suffices to prove that $\Sigma(\alpha) \models \tp^{\mathbb{F}}(\alpha)$ (we can then stipulate that the Zariski constructible sets $S$ are boolean combinations of elements of $\Sigma(\alpha)$ for $\alpha = (a^{1}, \lambda^{1}, \gamma^{1}, b^{1},c)$ in the proof of Proposition \ref{prop: core qe}. The proof of this proposition goes through, then apply Proposition \ref{prop: gen core qe}). By quantifier-elimination for $\mathbb{F}$ and noting that every constructible subset is a disjunction of conjuncts of the kind $C_{1} \wedge \neg C_{2}$, it remains to prove that $C_{1}, C_{2}$ (where $C_{1} \wedge \neg C_{2} \in tp^{\mathbb{F}}(\alpha)$) can be replaced with $\tilde{C}_{1}, \tilde{C}_{2}$ (respectively) such that $\tilde{C}_{2}$ is $\mu_{N}$-invariant and $(\tilde{C}_{1} \wedge \neg \tilde{C}_{2})\rightarrow (C_{1} \wedge \neg C_{2})$. Put
\[ \tilde{C}_{2} = \bigvee_{\delta \in \mu_{N}^{r_{2}}} C_{2}^{\delta} \] $\tilde{C}_{2}$ is closed, $\mu_{N}$-invariant and $\neg \tilde{C}_{2}$ implies $\neg C_{2}$. If $\neg \tilde{C}_{2} \in p = tp^{k}(\alpha)$ then we are done. Otherwise $\neg C_{2} \wedge \tilde{C}_{2} \in p$. Let $\Delta$ be the maximal (hence proper) subset of $\mu_{N}^{r_{2}}$ consisting of those $\delta$ such that 
\[ \neg D = \bigwedge_{\delta \in \Delta} \neg C_{2}^{\delta} \in p \] $\Delta$ is non-empty because $1 \in \Delta$. Put
\[ \Stab(\Delta) = \{\delta \in \mu_{N}^{r_{2}}: \delta \Delta = \Delta \} \] If $\delta \not \in \Stab(\Delta)$ then by maximality of $\Delta$ we have $\neg D^{\delta} \wedge \neg D \not \in p$, hence $D^{\delta}\in p$. Thus
\[ \bigwedge_{\delta \in \mu_{N}^{r_{2}}\setminus \Stab(\Delta)} D^{\delta} \in p \] \textbf{Claim}: We have 
\[ Q_N \models \bigwedge_{\delta \in \mu_{N}^{r_{2}} \setminus \Stab(\Delta)} D^{\delta} \wedge \bigvee_{\delta \in \mu_{N}^{r_{2}}} \neg D^{\delta} \rightarrow \bigvee_{\delta \in \Stab(\Delta)} \neg D^{\delta} \] 

\begin{proof} Suppose that $b \in \mathbb{F}$ is such that $D^{\delta}(b)$ holds for every $\delta \in \mu_{N}^{r_{2}} \setminus \Stab(\Delta)$ and $\neg D^{\delta_{1}}(b)$ holds for some $\delta_{1} \in \mu_{N}^{r_{2}}$. Then $\delta_{1} \in \Stab(\Delta)$ and the claim follows. 
\end{proof} The latter disjunct is clearly equivalent to $\neg D$ and $\neg D$ implies $\neg C_{2}$. So we take
\[ \tilde{C}_{1} = C_{1} \wedge \bigwedge_{\delta \in \mu_{N}^{r_{2}}\setminus \Stab(\Delta)} D^{\delta} \] and replace $\tilde{C}_{2}$ with $\bigwedge_{\delta \in \mu_{N}^{r_{2}}} D^{\delta}$. The result now follows. 
\end{proof}

\begin{lemma}
\label{lem: mu inv: replace} Let $\exists e (C_{1} 
\wedge \neg C_{2})$ be a general core formula with $C_{2}$ closed and $\mu_N$-invariant. Then there exists $\mu_N$-invariant $\tilde{C}_{1}$ such that
\[ \exists e (C_{1} \wedge \neg C_{2}) \equiv \exists e (\tilde{C}_{1} \wedge \neg C_{2}) \] 
\end{lemma}

\begin{proof} Define
\[ \tilde{C}_{1} = \bigvee_{\delta \in \mu_N^{r_2}} C_{1}^{\delta} \] It is immediate that $\exists e (C_{1} \wedge \neg C_{2})$ implies $\exists e (\tilde{C}_{1} \wedge \neg C_{2})$. Conversely, suppose that there is a tuple $(v,w,a)$ and $\lambda, \mu, e,y,\gamma,b$ such that 
\[ \bigwedge_{i=1}^{s} \bigwedge_{j=1}^{s_{i}} \pp(v_{ij}) = y_{i} \wedge \lambda_{ij}e_{i} = v_{ij} \wedge \mathbf{E}(e_{i}, y_{i}) \wedge \phi \wedge \bigwedge_{(i,j) \in \Xi} \phi_{ij} \wedge (C_{1}^{\delta} \wedge \neg C_{2})(\lambda, \mu, y, \gamma, b, a)\] holds for some $\delta$. Then 
\[ \bigwedge_{i=1}^{s} \bigwedge_{j=1}^{s_{i}} \pp(v_{ij}) = y_{i} \wedge \lambda'_{ij}e'_{i} = v_{ij} \wedge \mathbf{E}(e'_{i}, y_{i}) \wedge \phi \wedge \bigwedge_{(i,j) \in \Xi} \phi_{ij} \wedge (C_{1} \wedge \neg C_{2})(\lambda', \mu, y, \delta \cdot \gamma, b, a)\] holds by $\mu_{N}$-invariance of $C_{2}$, where $\lambda'_{ij} = \delta_{i}^{-1}\lambda_{ij}$. Now existentially quantify out $\lambda',\mu,e',y,\delta \cdot \gamma,b$. 
\end{proof}

\begin{remark} \label{rem: H-param} Suppose that $\exists e S$ is a general core formula with basis parameters $e' = (e'_{1}, \dots, e'_{p})$. Let $e'' = (e''_{1}, \dots, e''_{p})$ be another tuple of basis elements with $\mathbf{p}(e'_{i}) = \mathbf{p}(e''_{i})$ for every $i$. There is $\delta = (\delta_{1}, \dots, \delta_{p}) \in \mu_N^{p}$ such that $\delta_{i}e'_{i} = e''_{i}$ for every $i$. Then $\exists e S \equiv \exists e S'$ where $S'$ is obtained from $S$ by replacing the variables $\mu_{ij}$ by $\mu_{ij} \delta_{i}$, $\gamma_{ij}$ by $\gamma_{ij}\delta_{i}$ for $(i,j) \in \Xi_{1}$ and $\gamma_{ij}$ by $\gamma_{ij}\delta_{i}^{-1}$ for $(i,j) \in \Xi_{2}$. 
\end{remark}

\begin{lemma} 
\label{lem: gen core: boolean}
Let $\exists e C_{1}, \exists e C_{2}$ be general core formulas with the same enumeration of variables. Let $C_{1}, C_{2}$ be Zariski closed and suppose that $C_{2}$ $\mu_{N}$-invariant. Then 
\begin{enumerate}
\item $Q_N \models \exists e(C_{1} \wedge C_{2}) \leftrightarrow \exists e C_{1} \wedge \exists e C_{2}$
\item $Q_N \models \exists e (\neg C_{2}) \leftrightarrow \neg \exists e C_{2}$
\end{enumerate} 
\end{lemma}

\begin{proof} By Remark \ref{rem: H-param} we may assume that $\exists e C_{1}, \exists e C_{2}$ have the same basis parameters. 
\begin{enumerate}
\item Left-to-right is trivial. Conversely, if the right-hand side holds for a tuple $(v,w,a)$, then we may obtain different basis elements $e$ and $e'$ as witnesses to $\exists e C_{1}$ and $\exists e C_{2}$ respectively. But the $\mu_{N}$-invariance of $C_{2}$ means that we can transform $e'$ to $e$ without affecting validity. So the left-hand side holds. 
\item Right-to-left is easy. Conversely, suppose that $(v,w,a)$ satisfies $\exists e (\neg C_{2})$ and that $e$ is a tuple of basis elements witnessing this. If some basis elements $e'$ witness $\exists e C_{2}$ then we can transform $e'$ to $e$, and using the $\mu_{N}$-invariance of $C_{2}$ we get a contradiction. 
\end{enumerate}
\end{proof}

\begin{proposition} 
\label{prop: construct}
All definable subsets are constructible, namely every definable subset is a boolean combination of those defined by general core formulas $\exists e C$ where $C$ is Zariski closed and $\mu_{N}$-invariant. 
\end{proposition} 

\begin{proof} Immediate by Lemmas \ref{lem: construct}, \ref{lem: mu inv: replace} and \ref{lem: gen core: boolean}. \end{proof}

\section{Zariski Structure} 
We use the following definition of a Zariski structure.

\begin{definition} \label{defn: Zar} A {\bf (Noetherian) Zariski structure} is a tuple $(M, \tau_n, \dim: n \in \mathbb{Z}_{> 0})$ where 
\begin{itemize}
\item $\tau_n$ is a Noetherian topology on $M^{n}$ (satisfies the descending chain condition on closed sets) 
\item $\dim$ is a function which associates to each definable set a non-negative integer (we call this the {\em dimension} of the set)
\end{itemize} such that the following two sets of axioms are satisfied:
\\
\\
{\bf Topological} 
\begin{enumerate}
\item singletons of $M^{n}$ are closed, cartesian products of closed sets are closed, diagonals are closed, the image of a closed set under a permutation of coordinates is closed;
\item for $a \in M^{m}$ and $C$ a closed subset of $M^{m+n}$ the set
\[ C(a,M^{n}) = \{b \in M^{n}: (a,b) \in C \} \] is closed. 
\end{enumerate} \textbf{Dimension} 
\begin{enumerate}
\item singletons are $0$-dimensional
\item $\dim (X \cup Y) = \max\{\dim X, \dim Y\}$
\item closed irreducible subsets of locally closed irreducible subsets have strictly smaller dimension
\item let $X$ be an irreducible locally closed subset of $M^{m+n}$, $\pi: M^{m+n} \rightarrow M^{n}$ a projection onto the first $m$ coordinates. Then 
\[ \dim X = \dim \pi(X) + \min_{a \in \pi(X)} \dim (\pi^{-1}(a) \cap X) \] and there is an open subset $U$ of $\pi(X)$ such that
\[ \dim(\pi^{-1}(a) \cap X) = \min_{a \in \pi(X)}(\pi^{-1}(a) \cap X) \] 
\end{enumerate}
\end{definition} Syntactically, the topologies $\tau_n$ can be specified by introducing a predicate for each closed subset into the language. The reader is referred to \cite{Zil1} for more details. A closed subset is \textbf{irreducible} if it cannot be written as the union of two proper closed subsets. By Noetherianity, every closed set is a finite union of irreducible components. We also require the following additional properties of a Zariski structure $M$ may posess:
\begin{enumerate}
\item \textbf{Pre-smoothness}: for any constructible $X,Y \subseteq M^{n}$ and any irreducible component $Z$ of $X \cap Y$, $\dim Z \geq \dim X + \dim Y - n$
\item \textbf{Completeness}: projections of closed sets are closed. 
\end{enumerate}

\begin{remark} If $\dim$ is defined to be Krull dimension ($\dim C$ for $C$ closed irreducible is the maximal length of a chain of irreducible sets $C_0 \subset C_1 \subset \dots C_n = C$), $\dim M = 1$ and $M$ is presmooth, then Definition \ref{defn: Zar} is equivalent to the definition of a Zariski geometry in \cite{HZ} (see \cite{Zil1}, Section 3.3). 
\end{remark}

\subsection{Topology} 
We introduce a topology on $\mathcal{H}^{n} \times \mathbb{F}^{m}$ by taking as a basis of closed subsets those subsets that are defined by general core formulas $\exists e C(v,w,x)$ ($(v,w,x)$ a tuple of variables from $\mathcal{H}^{n} \times \mathbb{F}^{m}$) where $C$ is Zariski closed and $\mu_{N}$-invariant. Closed subsets are given by finite unions and arbitrary intersections of basic closed subsets. Note that if $n = 0$, then these formulas reduce to those of the form $C(x)$ where $C$ defines a Zariski closed subset of $\mathbb{F}^{m}$. Thus the topology  gives us the classical Zariski topology on the sort $\mathbb{F}$ and its cartesian powers. 

\begin{lemma}
\label{lem: gen core: equiv}
Let $\exists e C_{1}$, $\exists e C_{2}$ be general core formulas with $C_{1}, C_{2}$ Zariski closed and $\mu_N$-invariant. Suppose that both formulas have the same enumeration of variables. Then
\[ Q_N \models \exists e C_{1} \leftrightarrow \exists e C_{2} \Rightarrow Q_N \models C_{1} \leftrightarrow C_{2} \] 
\end{lemma}

\begin{proof} By Lemma \ref{lem: gen core: boolean}, $\exists e C_{1} \wedge \neg \exists e C_{2}$ is equivalent to $\exists e (C_{1} \wedge \neg C_{2})$ hence $C_{1} \wedge \neg C_{2}$ must be inconsistent. The rest of the lemma follows by symmetry. 
\end{proof} 

Although a general core formula $\exists e S$ was defined with respect to two tuples of variables $v = (v_{1}, \dots, v_{m})$ and $w = (w_{1}, \dots, w_{n})$, we shall henceforth amalgamate these into one tuple which we enumerate as $\{v_{ij}: 1 \leq i \leq s, 1 \leq j \leq s_{i} \}$ where there is $t \leq s$ for which $v_{ij} \in w$ for all $i > t$. 

\begin{proposition} 
\label{prop: Noetherian}
The topology defined on $Q_N$ is Noetherian. 
\end{proposition}

\begin{proof} Suppose for contradiction that $(\exists e C_{i}: i \in \mathbb{N})$ defines an infinite descending chain of basic closed subsets, i.e. we have proper inclusions $\exists e C_{i}(\mathcal{H},\mathbb{F}) \supset \exists e C_{i+1}(\mathcal{H}, \mathbb{F})$ for every $i$. Because there are only finitely many ways of enumerating the variables $v$ as $\{v_{ij}: 1 \leq i \leq s, 1 \leq j \leq s_{i}\}$, there are infinitely many $\exists e C_{i}$ with the same enumeration. Hence we can assume, without loss of generality, that all $\exists e C_{i}$ have the same enumeration of $v$ variables. By Lemma \ref{lem: gen core: boolean}, 
\[ \exists e C_{i+1}(\mathcal{H}, \mathbb{F}) = ( \exists e C_{i} \wedge \exists e C_{i+1})(\mathcal{H}, \mathbb{F}) = \exists e (C_{i} \wedge C_{i+1})(\mathcal{H}, \mathbb{F}) \] By Lemma \ref{lem: gen core: equiv} it follows that $C_{i}(\mathbb{F}) \supseteq C_{i+1}(\mathbb{F})$. Because $\exists e C_{i}(\mathcal{H}, \mathbb{F}) \supset \exists e C_{i+1}(\mathcal{H},\mathbb{F})$, Lemma \ref{lem: gen core: boolean} gives that $\exists e(C_{i} \wedge \neg C_{i+1})$ is satisfiable. Thus we have proper inclusions $C_{i}(\mathbb{F}) \supset C_{i+1}(\mathbb{F})$ for every $i$, contradicting that the Zariski topology is Noetherian. 
\end{proof}

\subsection{Zariski structure}
If $\exists e C$ be general core formula defining a basic closed set then $C$ is determined up to isomorphism by Lemma \ref{lem: gen core: equiv}. 

\begin{lemma} \label{lem: closed, irred} If $\exists e C$ is closed and irreducible, then it is basic closed. Moreover, $C = \bigvee D^{\delta}$ where $D$ defines a closed irreducible subset of $C(\mathbb{F})$. 
\end{lemma}

\begin{proof} The first statement is immediate. Let $C = \bigvee D_{i}$ where $D_{i}$ are irreducible components. Then $C = \bigvee_{i} \bigvee D_{i}^{\delta}$ by $\mu_N$-invariance of $C$, hence
\[ \exists e C \equiv \bigvee_{i} \exists e (\bigvee D_{i}^{\delta}) \] and by irreducibility, $\exists e C \equiv \exists e \bigvee D_{j}^{\delta}$ for some $j$. Now apply Lemma \ref{lem: gen core: equiv}.
\end{proof}

\begin{definition} 
\label{defn: dim} Let $\exists e C$ define a basic closed irreducible subset of $\mathcal{H}^{n} \times \mathbb{F}^{k}$. The \textbf{dimension} of $\exists e C(\mathcal{H}, \mathbb{F})$ is defined to be the dimension of $C(\mathbb{F})$. For $\exists e C$ defining a closed set, 
\[ \dim \exists e C := \max\{\dim C_{i}\} \] where $C_{i}$ are the irreducible components of $\exists e C$. If $\exists e S$ is constructible, its dimension is defined to be the dimension of its closure. 
\end{definition} By Proposition \ref{prop: construct} the projection of any constructible set is constructible. For definable sets defined by general core formulas we have more. 

\begin{lemma} \label{lem: gen core: proj}
Let $\exists e S$ be a general core formula with the aforementioned convention on enumeration of variables. For a fixed $1 \leq i \leq s$, let $j$ range over a subset $J \subseteq \{1, \dots, s_{i}\}$. Then $\exists_{j \in J} v_{ij} \exists e S$ is a general core formula with Zariski constructible component equivalent to one of the following:
\begin{enumerate}
\item $\exists_{j \in J} \lambda_{ij}S$. 
\item $\exists_{j \in J} \mu_{i-t,j}S$
\item $\exists \mu_{i-t,1} \exists_{(i-t,j) \in \Xi_{1}} b_{i-t,j} \exists_{(i-t,j) \in \Xi_{1}} \gamma_{i-t,j}  \exists_{(k,i-t)\in \Xi_{2}}  b_{k,i-t} \exists_{(k,i-t) \in \Xi_{2}} \gamma_{i-t,k} S$
\item $\begin{array}{l} \exists y_{i} \exists_{(i,k) \in \Xi} b_{ik} \exists_{(i,k) \in \Xi} \gamma_{ik} \exists_{(j,i) \in \Xi} b_{ji} \exists_{(j,i) \in \Xi} \gamma_{ji} \\ \exists_{(j-t,i) \in \Xi_{1}} b_{j-t,i} \exists_{(j-t,i) \in \Xi_{1}} \gamma_{j-t,i}  \exists_{(i,j-t)\in \Xi_{2}}  b_{i,j-t} \exists_{(i,j-t) \in \Xi_{2}} \gamma_{i,j-t} S \end{array}$
\end{enumerate}
\end{lemma}

\begin{proof} The proof divides into four cases:
\begin{enumerate}
\item $1 \leq i \leq t$.
\item $t+1 \leq i \leq s$. 
\item $t+1 \leq i \leq s$ and $s_{i} = 1$. 
\item $1 \leq i \leq t$ and $s_{i} = 1$. 
\end{enumerate} In case:

\begin{enumerate}
\item the $v_{ij}$ do not occur in $\phi$ and we can eliminate the conjuncts $\lambda_{ij}e_{i} = v_{ij}$, thus moving the quantifiers $\exists_{j \in J} \lambda_{ij}$ to $S$. 
\item $\exists v_{ij} \exists e S$ is equivalent to
\[ \exists \lambda \dots \exists b \left( \bigwedge_{i=1}^{t} \bigwedge_{j=1}^{s_{i}} \pi(v_{ij}) = y_{i} \wedge \lambda_{ij}e_{i} = v_{ij} \wedge \mathbf{E}(e_{i}, y_{i}) \wedge \exists v_{ij} \phi \wedge \bigwedge_{(i,j) \in \Xi} \phi_{ij} \wedge S(\lambda, \mu, y, \gamma, b, a)  \right) \] Recall that $\phi$ is 
\[ \begin{array}{ll} \bigwedge_{i=t+1}^{p} \bigwedge_{j=1}^{p_{i}} \mu_{i-t,j} e'_{i-t} = v_{ij} \wedge \bigwedge_{(i-t,j) \in \Xi_{1}} \phi_{i-t,j}(e'_{i-t}, e_{j}, b_{i-t,j}, \gamma_{i-t,j}) \wedge \\ \bigwedge_{(i,j-t) \in \Xi_{2}} \phi_{i-t,j}(e_{i}, e'_{j-t}, b_{i,j-t}, \gamma_{i,j-t}) \end{array} \] Thus $\exists v_{ij} \phi$ is equivalent to $\phi'$, where the latter is $\phi$ but with the conjuncts $\mu_{i-t,j} e'_{i-t} = v_{ij}$ removed for $j \in J$. It follows that we can move the quantifiers $\exists_{j \in J} \mu_{i-t,j}$ to $S$ as required. 
\item similar to $2$, but more is eliminable from $\phi$ because we can get rid of the parameter $e'_{i}$. Hence we can eliminate $\phi_{i-t,k}(e'_{i-t}, e_{k}, b_{i-t,k}, \gamma_{i-t,k})$ and $\phi_{k,i-t}(e_{k}, e'_{i-t}, b_{k,i-t}, \gamma_{k,i-t})$. The quantifiers
\[ \exists \mu_{i-t,1} \exists_{(i-t,j) \in \Xi_{1}} b_{i-t,j} \exists_{(i-t,j) \in \Xi_{1}} \gamma_{i-t,j}  \exists_{(k,i-t)\in \Xi_{2}}  b_{k,i-t} \exists_{(k,i-t) \in \Xi_{2}} \gamma_{k, i-t} \] can then be moved to $S$. 
\item the most is eliminable. We no longer require $\mathbf{E}(e_{i}, y_{i})$ and those conjuncts $\phi_{jk}$ with $(j,k) \in \Xi$ and $j$ or $k$ equal to $i$. But we can also eliminate conjuncts from $\phi$, namely $\phi_{j-t,i}$ for $(j-t,i) \in \Xi_{1}$ and $\phi_{i, j-t}$ for $(i, j-t) \in \Xi_{2}$. Thus we move the quantifiers
\[ \begin{array}{l} \exists y_{i} \exists_{(i,k) \in \Xi} b_{ik} \exists_{(i,k) \in \Xi} \gamma_{ik} \exists_{(j,i) \in \Xi} b_{ji} \exists_{(j,i) \in \Xi} \gamma_{ji} \\ \exists_{(j-t,i) \in \Xi_{1}} b_{j-t,i} \exists_{(j-t,i) \in \Xi_{1}} \gamma_{j-t,i}  \exists_{(i,j-t)\in \Xi_{2}}  b_{i,j-t} \exists_{(i,j-t) \in \Xi_{2}} \gamma_{i,j-t} \end{array} \] to $S$.
\end{enumerate}
\end{proof} 

\begin{theorem}
\label{thm: QN: Zar} $Q_N$ is a Zariski structure which is presmooth. 
\end{theorem}

\begin{proof} We have a Noetherian topology on $\mathcal{H}^{n} \times \mathbb{F}^{k}$ for each $n,k$ by Proposition \ref{prop: Noetherian}. The topological axioms are immediate as are the first three dimension axioms (our notion of dimension is derived from Krull dimension on $\mathbb{F}^{k}$). 4 and presmoothness are immediate from Lemma \ref{lem: gen core: proj} and the corresponding properties for algebraic varieties.  
\end{proof}

\begin{corollary} \label{cor: QHO: Zar} $\CA_N$ is an irreducible one-dimensional complete Zariski geometry. 
\end{corollary}

\begin{proof} $\CA_N$ is interpretable in $Q_N$ (Remark \ref{rem: bi-interp}) with the universe $L = \mathbf{E}$, $\A^{\dagger}$, $\A$ closed relations in the topology on $Q_N$. Now take the induced topological structure on $\CA_N$.
\end{proof}

\bibliographystyle{abbrv}
\bibliography{b}

\begin{thebibliography}{10}

\bibitem{birch-cycl}
B.~{Birch}.
\newblock Cyclotomic fields and {Kummer} extensions.
\newblock In J.~{Cassels} and A.~{Frohlich}, editors, {\em Algebraic Number
  Theory}. Academic press, 1973.

\bibitem{HZ}
E.~Hrushovski and B.~Zilber.
\newblock Zariski geometries.
\newblock {\em J. Amer. Math. Soc.}, 9(1):1--56, 1996.

\bibitem{moosa-nonstandard}
R.~Moosa.
\newblock A nonstandard {Riemann} existence theorem.
\newblock {\em Transactions of the American Mathematical Society},
  356(5):1781--1797, 2004.

\bibitem{pillay-ccm}
A.~Pillay.
\newblock Some model theory of compact complex spaces.
\newblock In {\em Workshop on {Hilbert's} tenth problem: relations with
  arithmetic and algbraic geometry}, volume 270 of {\em Contemporary
  Mathematics}, 2000.

\bibitem{pillay-acf}
A.~Pillay.
\newblock Model theory of algebraically closed fields.
\newblock In E.~Bouscaren, editor, {\em Model theory and algebraic geometry}.
  Springer-Verlag, 2002.

\bibitem{poizat-groups}
B.~Poizat.
\newblock {\em Stable groups}.
\newblock Number~87 in Mathematical Surveys and Monographs. American
  Mathematical Society, 2001.

\bibitem{serre-galois}
J.-P. Serre.
\newblock {\em Galois cohomology}.
\newblock Springer-Verlag, 1964.

\bibitem{weil-phs}
A.~Weil.
\newblock On algebraic groups and homogeneous spaces.
\newblock {\em American Journal of Mathematics}, pages 493--512, 1955.

\bibitem{z-ccm}
B.~Zilber.
\newblock Model theory and algebraic geometry.
\newblock In M.~Weese and H.~Walter, editors, {\em Proceedings of the 10th
  Easter conference on model theory}. Humboldt universit\"at, Berlin, 1993.

\bibitem{Zil2}
B.~Zilber.
\newblock A class of quantum {Zariski} geometries.
\newblock In Z.~Chatzidakis, H.~Macpherson, A.~Pillay, and A.~Wilkie, editors,
  {\em Model Theory with applications to algebra and analysis, I and II}.
  Cambridge University Press, 2008.

\bibitem{Zil1}
B.~Zilber.
\newblock {\em Zariski geometries: geometry from the logician's point of view}.
\newblock London Mathematical Society Lecture Note Series. Cambridge University
  Press, 2010.

\end{thebibliography}

\end{document}